\documentclass[final,3p,12pt]{elsarticle}
\usepackage{mathrsfs,enumerate,amssymb,amsmath,bm,color,lineno}
\usepackage[all]{xy}
\usepackage{latexsym}
\usepackage{amsthm,a4wide,color}
\usepackage{amscd,verbatim}
\usepackage{graphicx}
\usepackage[colorlinks=true]{hyperref}
\usepackage{amssymb}

\theoremstyle{plain}
\newtheorem{theorem}{Theorem}
\newtheorem{lemma}{Lemma}
\newtheorem{proposition}{Proposition}
\newtheorem{corollary}{Corollary}

\theoremstyle{definition}
\newtheorem{definition}{Definition}

\newtheorem{question}{Question}

\numberwithin{equation}{section}

\newtheorem{remark}{Remark}

\usepackage{pifont}

\newcommand{\wustar}{\text{\ding{76}}}

\journal{}

\begin{document}

\begin{frontmatter}
\title{On lattice-ordered $t_{r}$-norms for type-2 fuzzy sets\tnoteref{mytitlenote}}
\tnotetext[mytitlenote]{This work was supported by the National Natural Science Foundation of China
(No. 11601449), the Science and Technology Innovation Team of Education Department of
Sichuan for Dynamical System and its Applications (No. 18TD0013), the Youth Science and Technology
Innovation Team of Southwest Petroleum University for Nonlinear Systems (No. 2017CXTD02),
the Key Natural Science Foundation of Universities in Guangdong Province (No. 2019KZDXM027),
and the National Natural Science Foundation of China (Key Program) (No. 51534006).}

\author[a1,a2,a3]{Xinxing Wu\corref{mycorrespondingauthor}}
\cortext[mycorrespondingauthor]{Corresponding author}
\address[a1]{School of Sciences, Southwest Petroleum University, Chengdu, Sichuan 610500, China}
\address[a2]{Institute for Artificial Intelligence, Southwest Petroleum University, Chengdu, Sichuan 610500, China}
\address[a3]{Zhuhai College of Jilin University, Zhuhai, Guangdong 519041, China}
\ead{wuxinxing5201314@163.com}

\author[a4]{Guanrong Chen}
\address[a4]{Department of Electrical Engineering, City University of Hong Kong,
Hong Kong SAR, China}
\ead{gchen@ee.cityu.edu.hk}


\begin{abstract}
In this paper, it is proved that, for the truth value algebra of interval-valued fuzzy sets, the distributive laws
do not imply the monotonicity condition for the set inclusion operation. Then, a lattice-ordered $t_{r}$-norm,
which is not the convolution of $t$-norms on $[0, 1]$, is obtained. These results negatively answer  two open
problems posed by Walker and Walker in \cite{WW2006}.
\end{abstract}

\end{frontmatter}

\section{Introduction}

Throughout this paper, let $I=[0, 1]$, $I^{[2]}=\left\{[a, b]: 0\leq a\leq b\leq 1\right\}$,
and $Map(X, Y)$ be the set of all mappings from space $X$ to space $Y$. In particular, let $\mathbf{M}=Map(I, I)$.

To extend type-1 fuzzy sets (T1FSs), which are mappings from some universe to $I$,
and interval-valued fuzzy sets (IVFSs), which are mappings from some universe to $I^{[2]}$,
Zadeh~\cite{Z1975} introduced the notion of type-2 fuzzy sets (T2FSs) in 1975, which were then equivalently expressed in
different forms by Mendel et al. \cite{M2007,M2001,MJ2002}. Simply speaking, a T2FS is a mapping from a universe
to $Map(I, I)$.

\begin{definition}\cite{Z1965}
A {\it type-1 fuzzy set} $A$ in space $X$ is a mapping from $X$ to $I$, i.e.,
$A\in Map(X, I)$.
\end{definition}

\begin{definition}\cite{WW2005}
A {\it type-2 fuzzy set} $A$ in space $X$ is a mapping $A: X\rightarrow \mathbf{M}$,
i.e., $A\in Map(X, \mathbf{M})$.
\end{definition}

\begin{definition}\cite{WW2005}
A fuzzy set $A\in Map(X, I)$ is {\it normal} if $\sup\{A(x): x\in I\}=1$.
\end{definition}

\begin{definition}\cite{WW2005}
A function $f\in \mathbf{M}$ is {\it convex} if, for any $0\leq x\leq y\leq z\leq 1$, $f(y)\geq f(x)\wedge f(z)$.
\end{definition}

Let $\mathbf{N}$ and $\mathbf{L}$ denote the set of all normal functions in $\mathbf{M}$ and
the set of all normal and convex functions in $\mathbf{M}$, respectively.

For any subset $B$ of $X$, a special fuzzy set $\bm{1}_{B}$, called the {\it characteristic function}
of $B$, is defined by
$$
\bm{1}_{B}(x)=\begin{cases}
1, & x\in B,\\
0, & x\in X\backslash B.
\end{cases}
$$
Let
$\mathbf{J}=\{\bm{1}_{\{x\}}: x\in I\}$ and $\mathbf{K}=\{\bm{1}_{[a, b]}: 0\leq a \leq b\leq 1\}$.

As an extension of the logic connective conjunction and disjunction in classical two-valued logic,
triangular norms ($t$-norms) with the neutral $1$ and triangular conorms ($t$-conorms) with the neutral $0$
on $I$ were introduced by Menger \cite{Me1942} and by Schweizer and Sklar \cite{SS1961},
respectively. The $t$-norms for binary operations on $I^{[2]}$ were introduced by Castillo et al. \cite{GWW1996}.

\begin{definition}\cite{KMP2000,SS1961}
A binary operation $\ast: I^{2}\rightarrow I$ is a {\it $t$-norm} on $I$ if it satisfies the following axioms:
\begin{enumerate}
  \item[(T1)] (commutativity) $x\ast y=y\ast x$ for $x, y\in I$;
  \item[(T2)] (associativity) $(x \ast y) \ast z=x\ast (y\ast z)$ for $x, y, z\in I$;
  \item[(T3)] (increasing) $\ast$ is increasing in each argument;
  \item[(T4)] (neutral element) $1\ast x=x\ast 1=x$ for $x\in I$.
\end{enumerate}
A binary operation $\ast: I^2\rightarrow I$ is a
{\it $t$-conorm} on $I$ if it satisfies axioms (T1), (T2), and (T3) above;
and axiom (T4'): $0\ast x=x\ast 0=x$ for $x\in I$.
\end{definition}

\begin{definition}\cite[Definition~2]{WW2006}\cite[Definition~8]{GWW1996}\label{Def-2}
A binary operation $\vartriangle: I^{[2]}\times I^{[2]} \longrightarrow I^{[2]}$ is a {\it $t$-norm} on $I^{[2]}$ if,
for any $\bm{x}, \bm{y}, \bm{z}\in I^{[2]}$ and any $a, b\in I$ with $a\leq b$,
the following hold:
\begin{itemize}
\item[(1)] $[1, 1] \vartriangle \bm{x}=\bm{x}$;
\item[(2)] $\bm{x}\vartriangle \bm{y}=\bm{y}\vartriangle \bm{x}$;
\item[(3)] $(\bm{x}\vartriangle \bm{y})\vartriangle \bm{z}=\bm{x} \vartriangle (\bm{y}\vartriangle \bm{z})$;
\item[(4)] $\bm{x}\vartriangle (\bm{y}\vee \bm{z})=(\bm{x}\vartriangle \bm{y})\vee (\bm{x}\vartriangle \bm{z})$;
\item[(5)] $\bm{x}\vartriangle (\bm{y}\wedge \bm{z})=(\bm{x}\vartriangle \bm{y})\wedge (\bm{x}\vartriangle \bm{z})$;
\item[(6)] $[0, 1]\vartriangle [a, b]=[0, b]$;
\item[(7)] $[a, a]\vartriangle [b, b]=[c, c]$ for some $c\in I$;
\end{itemize}
where $[x_1, y_1]\wedge [x_2, y_2]=[x_1\wedge x_2, y_1 \wedge y_2]$,
and $[x_1, y_1]\vee [x_2, y_2]=[x_1\vee x_2, y_1\vee y_2]$.
\end{definition}

Walker and Walker~\cite{WW2006} proved that every $t$-norm $\vartriangle$ on $I^{[2]}$ is of the form
$[x_1, y_1]$ $\vartriangle$$[x_2, y_2]$ $=[x_1\blacktriangle x_2, y_1\blacktriangle y_2]$ for some $t$-norm $\blacktriangle$ on $I$,
and they introduced the following two monotonicity conditions to replace the distributive
laws (4) and (5):
\begin{itemize}
\item[(4$^{\prime}$)] $\bm{x}\leq \bm{y}$ implies $\bm{x}\vartriangle \bm{z}\leq \bm{y} \vartriangle \bm{z}$;
\item[(5$^{\prime}$)] $\bm{x}\subseteq \bm{y}$ implies $\bm{x}\vartriangle \bm{z}\subseteq \bm{y} \vartriangle \bm{z}$.
\end{itemize}
Meanwhile, they posed the following open problem:

\begin{question}\label{Q-1}\cite{WW2006}
Whether or not conditions (4) and (5) in Definition~\ref{Def-2} imply condition (5$^{\prime}$)?
\end{question}


A general technique to construct new operations on $\mathbf{M}$ is convolution.
\begin{definition}\cite[Definition 1.3.3]{HWW2016}\cite[Definition 6]{WW2006}
Let $\circ$ and $\vartriangle$ be two binary operations defined on $X$ and $Y$, respectively,
and $\Box$ be an appropriate operation on $Y$. If $\circ$ is a surjection, define a binary
operation $\bullet$ on the set $Map(X, Y)$ by
$$
(f\bullet g)(x)=\Box \{f(y) \vartriangle g(z): y\circ z=x\}.
$$
This method of defining a binary operation on $Map(X, Y)$ is called {\it convolution}.
In particular, the convolution of a $t$-norm $\vartriangle$ on $I$ is the binary operation
$\blacktriangle$ on $\mathbf{M}$ defined by
$$
(f\blacktriangle g)(x)=\sup\{f(y)\wedge g(z): y\vartriangle z=x\}, \text{ for } f, g\in \mathbf{M}.
$$
\end{definition}

\begin{definition}\cite{HCT2015}\label{HCT-Def}
Let $\ast$ be a binary operation on $I$, $\vartriangle$ be a
$t$-norm on $I$, and $\triangledown$ be a $t$-conorm on $I$. Define the binary operations
$\curlywedge$ and $\curlyvee: \mathbf{M}^2\rightarrow \mathbf{M}$ as follows: for $f, g\in \mathbf{M}$,
\begin{equation}\label{O-1}
(f\curlywedge g)(x)=\sup\left\{f(y)\ast g(z): y\vartriangle z =x\right\},
\end{equation}
and
\begin{equation}\label{O-2}
(f\curlyvee g)(x)=\sup\left\{f(y)\ast g(z): y\ \triangledown\ z =x\right\}.
\end{equation}
\end{definition}


\begin{definition}\cite{WW2005}\label{Def-7}
The operations of $\sqcup$ (union), $\sqcap$ (intersection), $\neg$ (complementation) on $\mathbf{M}$ are
defined as follows: for $f, g\in \mathbf{M}$,
\begin{align*}
(f\sqcup g)(x)&=\sup\{f(y)\wedge g(z): y\vee z=x\},\\
(f\sqcap g)(x)&=\sup\{f(y)\wedge g(z): y\wedge z=x\},
\end{align*}
and
$$
(\neg f)(x)=\sup\{f(y): 1-y=x\}=f(1-x).
$$
\end{definition}

From \cite{WW2005}, it follows that $\mathbb{M}=(\mathbf{M}, \sqcup, \sqcap, \neg, \bm{1}_{\{0\}}, \bm{1}_{\{1\}})$
is not a lattice, as the absorption laws do not hold, although $\sqcup$ and $\sqcap$ satisfy the De Morgan's laws with respect
to the complementation $\neg$.

Walker and Walker \cite{WW2005} defined the following partial order on $\mathbf{M}$.
\begin{definition}\cite{WW2005}
$f\sqsubseteq g$ if $f\sqcap g=f$; $f\preceq g$ if $f\sqcup g=g$.
\end{definition}

It is noted that the same orders were introduced by Mizumoto and Tanaka~\cite{MT1976} for $Map(J, I)$,
in the case that $J$ is a subset of $I$. It follows from \cite[Proposition 14]{WW2005} that $\sqsubseteq$
and $\preceq$ are different partial orders on $\mathbf{M}$. However, $\sqsubseteq$ and $\preceq$ coincide on $\mathbf{L}$,
and the lattice $(\mathbf{L}, \sqsubseteq)$ is a bounded complete lattice (see \cite{WW2005,HWW2007}). In particular, $\bm{1}_{\{0\}}$ and
$\bm{1}_{\{1\}}$ are the minimum and maximum of $\mathbf{L}$, respectively.

\begin{definition}\cite{HCT2015,WW2006}\label{Def-1}
A binary operation $T: \mathbf{L}^2 \rightarrow \mathbf{L}$ is a {\it $t_{r}$-norm}
({\it $t$-norm according to the restrictive axioms}), if
\begin{itemize}
  \item[(O1)] $T$ is commutative, i.e., $T(f, g)=T(g, f)$ for $f, g\in \mathbf{L}$;
  \item[(O2)] $T$ is associative, i.e., $T(T(f, g), h)=T(f, T(g, h))$ for $f, g, h\in \mathbf{L}$;
  \item[(O3)] $T(f, \bm{1}_{\{1\}})=f$ for $f\in \mathbf{L}$ (neutral element);
  \item[(O4)] for $f, g, h\in \mathbf{L}$ such that $f\sqsubseteq g$,
  $T(f, h)\sqsubseteq T(g, h)$ (increasing in each argument);
  \item[(O5)] $T(\bm{1}_{[0, 1]}, \bm{1}_{[a, b]})=\bm{1}_{[0, b]}$;
  \item[(O6)] $T$ is closed on $\mathbf{J}$;
  \item[(O7)] $T$ is closed on $\mathbf{K}$.
\end{itemize}

A binary operation $S: \mathbf{L}^2\rightarrow \mathbf{L}$ is a
{\it $t_r$-conorm} if it satisfies axioms (O1), (O2), (O4), (O6), and (O7) above,
axiom (O3$^{'}$): $S(f, \bm{1}_{\{0\}})=f$, and axiom (O5$^{'}$):
$S(\bm{1}_{[0, 1]}, \bm{1}_{[a, b]})=\bm{1}_{[a, 1]}$. Axioms (O1), (O2), (O3), (O3$^{'}$),
and (O4) are called ``{\it basic axioms}", and an operation that complies with
these axioms will be referred to as {\it $t$-norm} and {\it $t$-conorm}, respectively.
\end{definition}

\begin{definition}\cite[Definition~5.2.6]{HWW2016}
A binary operation $R: \mathbf{L}^2\rightarrow \mathbf{L}$ is a
{\it lattice-ordered $t_{r}$-norm} (denoted as {\it $t_{lor}$-norm}) if it satisfies axioms
(O1), (O2), (O3), (O5), (O6), and (O7) above,
axiom (O4$^{'}$): $R(f, g\sqcup h)=R(f, g)\sqcup R(f, h)$, and axiom (O4$^{''}$):
$R(f, g\sqcap h)=R(f, g)\sqcap R(f, h)$.
\end{definition}

\begin{remark}
Recently, we \cite{WC-T-Norm} proved that $t_{lor}$-norm on $\mathbf{L}$ is strictly stronger
than ${t_r}$-norm on $\mathbf{L}$, which is strictly stronger than ${t}$-norm on $\mathbf{L}$.
\end{remark}

Walker and Walker~\cite{WW2006} proved that the convolution $\blacktriangle$ of each $t$-norm
$\vartriangle$ on $I$ is a $t_{lor}$-norm on $\mathbf{L}$ and they proposed the following question
in \cite{WW2006}.

\begin{question}\label{Q-2}\cite{WW2006}
Whether or not a $t_{lor}$-norm is indeed the convolution of a $t$-norm on $I$?
\end{question}

Hern\'{a}ndez~et al. \cite{HCT2015} proved that the binary operations
$\curlywedge$ and $\curlyvee$, defined in Definition \ref{HCT-Def}, are respectively a $t_r$-norm and
a $t_r$-conorm on $\mathbf{L}$, provided that $\vartriangle$ and
$\triangledown$ are continuous and $\ast$ is a continuous $t$-norm on $I$. Concerning its converse,
we \cite{WC-TFS,WC-T-Norm} showed that if the operation $\curlywedge$ defined in Definition
\ref{HCT-Def} is a $t_r$-norm on $\mathbf{L}$, then $\vartriangle$ is continuous
and $\ast$ is a $t$-norm on $I$, and we also obtained a similar result for $\curlyvee$. Meanwhile, we \cite{WC-TFS-2}
constructed a $t_r$-norm and a $t_r$-conorm on $\mathbf{L}$, which cannot be obtained by the
formulas that define the operations `$\curlywedge$' and `$\curlyvee$'.

Extending our construction method in \cite{WC-TFS-2}, this paper is devoted to answering Questions~\ref{Q-1} and \ref{Q-2}.
In Section~\ref{S-iii}, we construct a binary operation $\circledast$ on $I^{[2]}$ satisfying conditions (4)
and (5) in Definition~\ref{Def-2}, which does not satisfy condition (5$^{\prime}$). In Sections~\ref{S-iV} and \ref{S-V},
we obtain a $t_{lor}$-norm $\wustar$, which is not the convolution of each $t$-norm on $I$. These two results negatively answer Questions
\ref{Q-1} and \ref{Q-2}.

\section{Some basic properties of $\mathbf{L}$}

\begin{definition}\cite{WW2005,WC-TFS-2}
For $f\in \mathbf{M}$, define 
\begin{align*}
f^{L}(x)&=\sup\left\{f(y): y\leq x\right\},\\
f^{L_{\mathrm{w}}}(x)&=\begin{cases}
\sup\{f(y): y< x\}, & x\in (0, 1], \\
f(0), & x=0,
\end{cases}
\end{align*}
and
\begin{align*}
f^{R}(x)&=\sup\left\{f(y): y\geq x\right\},\\
f^{R_{\mathrm{w}}}(x)&=\begin{cases}
\sup\{f(y): y> x\}, & x\in [0, 1), \\
f(1), & x=1.
\end{cases}
\end{align*}
\end{definition}
Clearly, (1) $f^L$, $f^{L_{\mathrm{w}}}$ and $f^R$, $f^{R_{\mathrm{w}}}$ are monotonically increasing and
decreasing, respectively; (2) $f^{L}(x)\vee f^{R}(x)=f^{L}(x)\vee f^{R_{\mathrm{w}}}(x)=
f^{R}(x)\vee f^{L_{\mathrm{w}}}(x)=\sup_{x\in I}\{f(x)\}$ for all $x\in I$. The following properties of $f^{L}$
and $f^{R}$ are obtained by Walker {\it et al.} \cite{WW2005}.

\begin{proposition}{\rm \cite{WW2005}}\label{WW-L-function}
For $f, g\in \mathbf{M}$,
\begin{enumerate}[{\rm (1)}]
  \item $f\leq f^{L}\wedge f^{R}$;
  \item $(f^{L})^{L}=f^{L}$, $(f^{R})^{R}=f^{R}$;
  \item $(f^{L})^{R}=(f^{R})^{L}=\sup_{x\in I} \{f(x)\}$;
  \item $f\sqsubseteq g$ if and only if $f^{R}\wedge g \leq f\leq g^{R}$;
  \item $f\preceq g$ if and only if $f\wedge g^{L}\leq g\leq f^{L}$;
  \item $f$ is convex if and only if $f=f^{L}\wedge f^{R}$.
\end{enumerate}
\end{proposition}

\begin{lemma}{\rm \cite{WW2005}}\label{L-R-Operation}
For $f, g\in \mathbf{L}$,
\begin{enumerate}[{\rm (i)}]
  \item $(f\sqcap g)^{L}=f^{L}\vee g^{L}$;
  \item $(f\sqcap g)^{R}=f^{R}\wedge g^{R}$;
  \item $(f\sqcup g)^{L}=f^{L}\wedge g^{L}$;
  \item $(f\sqcup g)^{R}=f^{R}\vee g^{R}$.
\end{enumerate}
\end{lemma}

\begin{remark}\label{remark-1}
From Lemma~\ref{L-R-Operation}, it follows that, for $f, g\in \mathbf{L}$,
$(f\sqcap g)(1)=(f\sqcap g)^{R}(1)=f(1)\wedge g(1)$ and $(f\sqcup g)(1)=(f\sqcup g)^{R}(1)
=f(1)\vee g(1)$.
\end{remark}

%

\begin{theorem}{\rm \cite{HWW2010,HWW2008}}\label{order-theorem}
Let $f, g\in \mathbf{L}$. Then, $f\sqsubseteq g$ if and only if
$g^L\leq f^L$ and $f^R\leq g^R$.
\end{theorem}

%

\begin{proposition}\label{Pro-L}
For $f\in \mathbf{M}$, it holds that $f^{L_{\mathrm{w}}}(x)=\sup_{t\in [0, x)}\{f^{L}(t)\}=\lim_{t\nearrow x}f^{L}(t)$ for all $x\in (0, 1]$.
\end{proposition}

\begin{proof}
Fix any $x\in (0, 1]$, noting that $f(t)\leq f^{L}(t)$ for all $t\in [0, x)$, one has
$$
f^{L_{\mathrm{w}}}(x)=\sup_{t\in [0, x)}\{f(t)\}\leq\sup_{t\in [0, x)}\left\{f^{L}(t)\right\}.
$$

Meanwhile, for any $t\in [0, x)$, it follows from $t< \frac{t+x}{2}<x$ that
$f^{L}(t)\leq f^{L_{\mathrm{w}}}(\frac{t+x}{2})\leq f^{L_{\mathrm{w}}}(x)$.
This implies that
$$
\sup_{t\in [0, x)}\left\{f^{L}(t)\right\}\leq f^{L_{\mathrm{w}}}(x).
$$
Thus,
$$
f^{L_{\mathrm{w}}}(x)=\sup_{t\in [0, x)}\left\{f^{L}(t)\right\}.
$$
\end{proof}

\begin{corollary}\label{f^Lw}
Let $f, g\in \mathbf{L}$. Then, for any $x\in (0, 1]$,
\begin{itemize}
  \item[{\rm (1)}] $(f\sqcap g)^{L_{\mathrm{w}}}(x)=f^{L_{\mathrm{w}}}(x)\vee g^{L_{\mathrm{w}}}(x)$,
  \item[{\rm (2)}] $(f\sqcup g)^{L_{\mathrm{w}}}(x)=f^{L_{\mathrm{w}}}(x)\wedge g^{L_{\mathrm{w}}}(x)$.
\end{itemize}
\end{corollary}
\begin{proof}
(1) Applying Lemma~\ref{WW-L-function} and Proposition~\ref{Pro-L} yields that
\begin{align*}
(f\sqcap g)^{L_{\mathrm{w}}}(x)&=\sup_{t\in [0, x)}\{(f\sqcap g)^{L}(t)\}
=\sup_{0\leq t<x}\{f^{L}(t)\vee g^{L}(t)\}\\
&\geq \sup_{0\leq t<x}\{f(t)\vee g(t)\}
\geq f^{L_{\mathrm{w}}}(x)\vee g^{L_{\mathrm{w}}}(x).
\end{align*}
Clearly, $\sup_{0\leq t<x}\{f^{L}(t)\vee g^{L}(t)\}\leq f^{L_{\mathrm{w}}}(x)\vee g^{L_{\mathrm{w}}}(x)$.
Then,
$$
(f\sqcap g)^{L_{\mathrm{w}}}(x)=f^{L_{\mathrm{w}}}(x)\vee g^{L_{\mathrm{w}}}(x).
$$

(2) Applying Lemma~\ref{WW-L-function} and Proposition~\ref{Pro-L} yields that
\begin{equation}\label{eq-Cor-1}
(f\sqcup g)^{L_{\mathrm{w}}}(x)=\sup_{t\in [0, x)}\{(f\sqcup g)^{L}(t)\}
=\sup_{0\leq t<x}\{f^{L}(t)\wedge g^{L}(t)\}\leq f^{L_{\mathrm{w}}}(x)\wedge g^{L_{\mathrm{w}}}(x).
\end{equation}
Let $\sup_{0\leq t<x}\{f^{L}(t)\wedge g^{L}(t)\}=\xi$.
Since $f^{L}\wedge g^{L}$ is increasing, it follows that, for $t_n\nearrow x$,
$f^{L}(t_n)\wedge g^{L}(t_n)\nearrow \xi$. Set $\mathscr{P}=\{n\in \mathbb{N}: f^{L}(t_n)\leq g^{L}(t_n)\}$
and $\mathscr{Q}=\{n\in \mathbb{N}: f^{L}(t_n)> g^{L}(t_n)\}$. Clearly, either $\mathscr{P}$ or $\mathscr{Q}$
is infinite.
\begin{itemize}
\item[(a)] If $\mathscr{P}$ is infinite, then
$$
\xi=\lim_{\mathscr{P}\ni n\to +\infty}f^{L}(t_n)\wedge g^{L}(t_n)=\lim_{\mathscr{P}\ni n\to +\infty}f^{L}(t_n)
=f^{L_{\mathrm{w}}}(x)\geq f^{L_{\mathrm{w}}}(x)\wedge g^{L_{\mathrm{w}}}(x),
$$
which, together with \eqref{eq-Cor-1}, implies that
$$
(f\sqcup g)^{L_{\mathrm{w}}}(x)=f^{L_{\mathrm{w}}}(x)\wedge g^{L_{\mathrm{w}}}(x).
$$
\item[(b)] If $\mathscr{Q}$ is infinite, then
$$
\xi=\lim_{\mathscr{Q}\ni n\to +\infty}f^{L}(t_n)\wedge g^{L}(t_n)=\lim_{\mathscr{Q}\ni n\to +\infty}g^{L}(t_n)
=g^{L_{\mathrm{w}}}(x)\geq f^{L_{\mathrm{w}}}(x)\wedge g^{L_{\mathrm{w}}}(x),
$$
which, together with \eqref{eq-Cor-1}, implies that
$$
(f\sqcup g)^{L_{\mathrm{w}}}(x)=f^{L_{\mathrm{w}}}(x)\wedge g^{L_{\mathrm{w}}}(x).
$$
\end{itemize}
\end{proof}

For $f\in \mathbf{L}$ and $\alpha\in I$, let $L_{\alpha}(f)=\inf\{x\in I: f^{L}(x)\geq \alpha\}$
and $R^{\alpha}(f)=\sup\{x\in I: f^{R}(x)\geq \alpha\}$.

\begin{lemma}\label{<-Lemma}{\rm \cite[Lemma~2]{WC-TFS-2}}
For $f\in \mathbf{N}$, $L_{1}(f)\leq R^{1}(f)$.
\end{lemma}


\begin{definition}
For $f\in \mathbf{M}$, let
$$
b_{f}=\sup\{x\in I: f(x)=f^{L}(x)\} \text{ and } c_{f}=\inf\{x\in I: f(x)=f^{R}(x)\}.
$$
The points $b_{f}$ and $c_{f}$ are the {\it left balance point} of $f$ and
{\it right balance point} of $f$, respectively. In \cite{HWW2007}, the point $b_{f}$ is also called
{\it balance point} of $f$.
\end{definition}

\begin{proposition}
For $f\in \mathbf{L}$, $b_{f}=R^{1}(f)$ and $L_{1}(f)=c_{f}$.
\end{proposition}
\begin{proof}
Let $A=\{x\in I: f^{R}(x)=1\}$ and $B=\{x\in I: f(x)=f^{L}(x)\}$. For $x\in A$, one has $f(x)=f^{L}(x)\wedge f^{R}(x)=f^{L}(x)$,
implying that $A\subset B$. For $x\in B$, one has $f(x)=f^{L}(x)=f^{L}(x)\wedge f^{R}(x)$, implying that $f^{R}(x)\geq f^{L}(x)$.
This, together with the normality of $f$, implies that $1=f^{L}(x)\vee f^{R}(x)=f^{R}(x)$, which means that $B\subset A$.
Thus, $b_{f}=\sup A=\sup B=R^{1}(f)$. $L_{1}(f)=c_{f}$ can be verified similarly.
\end{proof}

\begin{proposition}\label{f^l--f^r=1}
Let $f\in \mathbf{N}$. Then,
\begin{enumerate}[{\rm (1)}]
\item $f^{L}(x)=1$ for $x\in \left(L_{1}(f), 1\right]$,
\item $f^{R}(x)=1$ for $x\in \left[0, R^{1}(f)\right)$.
\end{enumerate}
\end{proposition}

%

\begin{lemma}\label{sup{f^R}}
Let $f, g\in \mathbf{L}$. Then,
\begin{itemize}
\item[{\rm (1)}] $L_{1}(f\sqcap g)= L_{1}(f)\wedge L_{1}(g)$, i.e., $\inf\{x\in I: (f\sqcap g)^{L}(x)=1\}=\inf\{x\in I: f^{L}(x)=1\}
\wedge \inf\{x\in I: g^{L}(x)=1\}$;
\item[{\rm (2)}] $R^{1}(f\sqcap g)= R^{1}(f)\wedge R^{1}(g)$, i.e., $\sup\{x\in I: (f\sqcap g)^{R}(x)=1\}=\sup\{x\in I: f^{R}(x)=1\}
\wedge \sup\{x\in I: g^{R}(x)=1\}$;
\item[{\rm (3)}] $L_{1}(f\sqcup g)=L_{1}(f)\vee L_{1}(g)$, i.e., $\inf\{x\in I: (f\sqcup g)^{L}(x)=1\}=\inf\{x\in I: f^{L}(x)=1\}
\vee \inf\{x\in I: g^{L}(x)=1\}$;
\item[{\rm (4)}] $R^{1}(f\sqcup g)=R^{1}(f)\vee R^{1}(g)$, i.e., $\sup\{x\in I: (f\sqcup g)^{R}(x)=1\}=\sup\{x\in I: f^{R}(x)=1\}
\vee \sup\{x\in I: g^{R}(x)=1\}$.
\end{itemize}
\end{lemma}
\begin{proof}
For convenience, denote $\inf\{x\in I: f^{L}(x)=1\}=\eta_1$ and $\inf\{x\in I: g^{L}(x)=1\}=\eta_2$. Clearly,
\begin{equation}\label{Eq-3}
(\eta_1, 1] \subset \{x\in I: f^{L}(x)=1\} \subset [\eta_1, 1],
\end{equation}
and
\begin{equation}\label{Eq-4}
(\eta_2, 1] \subset \{x\in I: g^{L}(x)=1\} \subset [\eta_2, 1].
\end{equation}
Applying Lemma~\ref{L-R-Operation} yields that
\begin{align*}
&\quad\ \{x\in I: (f\sqcap g)^{L}(x)=1\}\\
&=\{x\in I: f^{L}(x)\vee g^{L}(x)=1\}\\
&=\{x\in I: f^{L}(x)=1\}\\
& \quad\ \cup \{x\in I: g^{L}(x)=1\}.
\end{align*}
This, together with (\ref{Eq-3}) and (\ref{Eq-4}), implies that
$$
(\eta_1\wedge \eta_2, 1]\subset \{x\in I: (f\sqcap g)^{L}(x)=1\}
\subset [\eta_1\wedge \eta_2, 1].
$$
Thus,
$$
\inf\{x\in I: (f\sqcap g)^{L}(x)=1\}=\eta_1\wedge \eta_2.
$$
The rest can be verified similarly.
\end{proof}

\begin{proposition}\label{convex-cha}
For $f\in \mathbf{L}$,
$$
f(x)=
\begin{cases}
f^{L}(x), & x\in [0, \xi_1), \\
f(\xi_1), & x=\xi_1, \\
1, & x\in (\xi_1, \xi_2), \\
f(\xi_2), & x=\xi_2, \\
f^{R}(x), & x\in (\xi_2, 1],
\end{cases}
$$
where $\xi_1=L_1(f)$ and $\xi_2=R^{1}(f)$.
\end{proposition}

\begin{proof}
Since $f$ is convex, from Proposition~\ref{WW-L-function}, it follows that $f=f^{L}\wedge f^{R}$. Consider the
following three cases:
\begin{itemize}
  \item[Case 1.] If $x\in [0, \xi_1)$, from Lemma~\ref{<-Lemma}, it follows that $x<\xi_1\leq \xi_2$. This implies that
$f^{R}(x)=1$. Thus, $f(x)=f^{L}(x)\wedge f^{R}(x)=f^{L}(x)$;
  \item[Case 2.] If $x\in (\xi_1, \xi_2)$, form the choices of $\xi_1$ and $\xi_2$, it can be verified that $f^{L}(x)=f^{R}(x)=1$.
  This implies that $f(x)=f^{L}(x)\wedge f^{R}(x)=1$;
  \item[Case 3.] If $x\in (\xi_2, 1]$, from Lemma~\ref{<-Lemma}, it follows that $\xi_1\leq \xi_2<x$. This implies that
$f^{L}(x)=1$. Thus, $f(x)=f^{L}(x)\wedge f^{R}(x)=f^{R}(x)$.
\end{itemize}
\end{proof}

\begin{remark}\label{remark-2}
From Proposition~\ref{convex-cha}, it follows that
\begin{enumerate}[(1)]
  \item every function $f$ in $\mathbf{L}$ is increasing on $[0, L_{1}(f))$,
constant on $(L_{1}(f), R^{1}(f))$, and decreasing on $(R^{1}(f), 1]$;
  \item if $L_{1}(f)<R^{1}(f)$, one has
  $f(L_{1}(f))\geq f^{L_{\mathrm{w}}}(L_{1}(f))$, i.e.,
  $f(L_{1}(f))\geq f(x)$ for all $x\in [0, L_{1}(f))$, and $f(R^{1}(f))\geq f^{R_{\mathrm{w}}}(R^{1}(f))$.
  \end{enumerate}
\end{remark}

\begin{corollary}\label{f^{R}}
Let $f\in \mathbf{L}$ satisfy that $R^{1}(f)<1$. Then, for $x\in [0, 1)$, $f^{R}(x)=\sup\{f(y): y\in [x, 1)\}$.
\end{corollary}

\begin{proof}
For $x\in [0, 1)$, choose $\zeta\in (R^{1}(f), 1)$ such that $\zeta>x$. From Proposition~\ref{convex-cha}, it follows that
$f(\zeta)\geq f(1)$, implying that
$$
f^{R}(x)=\sup\{f(y): x\leq y\leq 1\}=\sup\{f(y): x\leq y<1\}.
$$
\end{proof}

\section{A negative answer to Question~\ref{Q-1}}\label{S-iii}

This section constructs a binary operation $\circledast$ on $I^{[2]}$ satisfying conditions (4)
and (5), which does not satisfy condition (5$^{\prime}$), answering negatively Question~\ref{Q-1}.
\begin{proposition}\label{Min-Single}
Let $[a_1, b_1], [a_2, b_2]\subset I$. If $[a_1 \wedge a_2, b_1 \wedge b_2]$
is a single point, then one of $[a_1, b_1]$ and $[a_2, b_2]$ is a single point.
\end{proposition}
\begin{proof}
Consider the following two cases:
\begin{itemize}
\item[Case 1.] If $a_2\leq a_1$, then $[a_1 \wedge a_2, b_1 \wedge b_2]=[a_2, b_1 \wedge b_2]$
is a single point, i.e., $b_1\wedge b_2=a_2$. This implies that
$b_1=a_2$ or $b_2=a_2$. When $b_1=a_2$, one has
$[a_1, b_1]$ is a single point, since $[a_1, b_1]\subset [a_2, b_1]$. When
$b_2=a_2$, one has $[a_2, b_2]$ is a single point.

\item[Case 2.] If $a_1<a_2$, then $[a_1\wedge a_2, b_1\wedge b_2]
=[a_1, b_1\wedge b_2]$ is a single point, i.e., $a_1=b_1\wedge b_2$. This, together with
$b_2\geq a_2>a_1$, implies that $b_1\wedge b_2=b_1=a_1$. This means that
$[a_1, b_1]$ is a single point.
\end{itemize}
\end{proof}

\begin{proposition}\label{Max-Single}
Let $[a_1, b_1], [a_2, b_2]\subset I$. If $[a_1\vee a_2, b_1\vee b_2]$
is a single point, then one of $[a_1, b_1]$ and $[a_2, b_2]$ is a single point.
\end{proposition}
\begin{proof}
Consider the following two cases:
\begin{itemize}
\item[Case 1.] If $a_2\leq a_1$, then $[a_1\vee a_2, b_1\vee b_2]
=[a_1, b_1\vee b_2]$ is a single point. This, together with $[a_1, b_1\vee b_2]
=[a_1, b_1]\cup [a_1, b_2]$, implies that $[a_1, b_1]$ is a single point.

\item[Case 2.] If $a_1<a_2$, then $[a_1\vee a_2, b_1\vee b_2]
=[a_2, b_1\vee b_2]$ is a single point. This, together with $[a_2, b_1\vee b_2]
=[a_2, b_1]\cup [a_2, b_2]$, implies that $[a_2, b_2]$ is a single point.
\end{itemize}
\end{proof}

%
%

\begin{definition}\label{xinxing-operation}
Define a binary operation $\circledast$ on $I^{[2]}$ as follows: for $\bm{x}, \bm{y}\in I^{[2]}$,
$$
\bm{x}\circledast \bm{y}=\max\{x\cdot y : x\in \bm{x}, y\in \bm{y}\}.
$$
\end{definition}

\begin{proposition}\label{Pro-Cond-5'}
The binary operation $\circledast$ defined in Definition~\ref{xinxing-operation} does not
satisfy condition (5$^{\prime}$).
\end{proposition}

\begin{proof}
Take $\bm{x}=\{0.5\}$ and $\bm{y}=[0.5, 1]$. Clearly, $\bm{x}\subset \bm{y}$.
Let $\bm{z}=\{0.5\}\in I^{[2]}$. From Definition~\ref{xinxing-operation}, it can be verified that
$$
\bm{z}\circledast \bm{x}=\{0.25\},
$$
and
$$
\bm{z}\circledast\bm{y}=\{0.5\}.
$$
Clearly, $\bm{z}\circledast \bm{x}\nsubseteq \bm{z}\circledast\bm{y}$.
Therefore, $\circledast$ does not satisfy condition 5$^{\prime}$.
\end{proof}

\begin{proposition}\label{Pro-Cond-4}
The binary operation $\circledast$ defined in Definition~\ref{xinxing-operation}
satisfies condition (4) in Definition~\ref{Def-2}.
\end{proposition}

\begin{proof}
For $\bm{x}=[x_1, x_2], \bm{y}=[y_1, y_2], \bm{z}=[z_1, z_2]\in I^{[2]}$,
one has
$$
\bm{x}\circledast \bm{y}=[x_1, x_2]\circledast [y_1, y_2]=\{x_2\cdot y_2\},
$$
$$
\bm{x}\circledast \bm{z}=[x_1, x_2]\circledast [z_1, z_2]=\{x_2\cdot z_2\},
$$
and
$$
\bm{x}\circledast (\bm{y}\vee \bm{z})=[x_1, x_2]\circledast [y_1\vee z_1, y_2\vee z_2]
=\{x_2\cdot (y_2\vee z_2)\},
$$
implying that
\begin{align*}
(\bm{x}\circledast \bm{y})\vee (\bm{x}\circledast \bm{z})&=
\{(x_2\cdot y_2)\vee (x_2\cdot z_2)\}\\
&=\{x_2\cdot (y_2\vee z_2)\}=
\bm{x}\circledast (\bm{y}\vee \bm{z}).
\end{align*}
\end{proof}

\begin{proposition}\label{Pro-Cond-5}
The binary operation $\circledast$ defined in Definition~\ref{xin-operation}
satisfies condition (5) in Definition~\ref{Def-2}.
\end{proposition}

\begin{proof}
For $\bm{x}=[x_1, x_2], \bm{y}=[y_1, y_2], \bm{z}=[z_1, z_2]\in I^{[2]}$,
one has
$$
\bm{x}\circledast \bm{y}=[x_1, x_2]\circledast [y_1, y_2]=\{x_2\cdot y_2\},
$$
$$
\bm{x}\circledast \bm{z}=[x_1, x_2]\circledast [z_1, z_2]=\{x_2\cdot z_2\},
$$
and
$$
\bm{x}\circledast (\bm{y}\wedge \bm{z})=[x_1, x_2]\circledast [y_1\wedge z_1, y_2\wedge z_2]
=\{x_2\cdot (y_2\wedge z_2)\},
$$
implying that
\begin{align*}
(\bm{x}\circledast \bm{y})\wedge (\bm{x}\circledast \bm{z})&=
\{(x_2\cdot y_2)\wedge (x_2\cdot z_2)\}\\
&=\{x_2\cdot (y_2\wedge z_2)\}=
\bm{x}\circledast (\bm{y}\wedge \bm{z}).
\end{align*}
\end{proof}

\begin{remark}
Summing up Propositions~\ref{Pro-Cond-5'}--\ref{Pro-Cond-5}, it follows that conditions (4) and (5) in
Definition~\ref{Def-2} do not imply condition (5$^{\prime}$). This gives a negative answer to Question~\ref{Q-1}.
\end{remark}

\section{Construct a $t_{lor}$-norm `$\wustar$' on $\mathbf{L}$}\label{S-iV}
Modifying our construction method in \cite{WC-TFS-2}, this section introduces a binary
operation `$\wustar$' on $\mathbf{L}$ and proves
that it is indeed a $t_{lor}$-norm.

\begin{definition}
Define a binary operation $\barwedge: \mathbf{M}^{2}\longrightarrow \mathbf{M}$
as follows: for $f, g\in \mathbf{M}$,
$$
(f\barwedge g)(x)=
\begin{cases}
(f\sqcap g)(x), & x\in [0, 1), \\
0, & x=1.
\end{cases}
$$
\end{definition}

\begin{definition}\label{bingstar-operation-2}
Define a binary operation $\wustar: \mathbf{L}^2\rightarrow \mathbf{M}$ as follows: for $f, g\in \mathbf{L}$,

Case 1. $f=\bm{1}_{\{1\}}$, $f\wustar g=g\wustar f=g$;

Case 2. $g=\bm{1}_{\{1\}}$, $f\wustar g=g\wustar f=f$;

Case 3. $f\neq \bm{1}_{\{1\}}$ and $g\neq \bm{1}_{\{1\}}$,
\begin{equation}\label{xin-operation}
f\wustar g=\begin{cases}
f\sqcap g, & f(1)\wedge g(1)=1,\\
f\barwedge g, & f(1)\wedge g(1)<1.
\end{cases}
\end{equation}
\end{definition}

\begin{remark}\label{Remark-bingstar}
From Definition~\ref{bingstar-operation-2} and Remark~\ref{remark-1}, it can be verified that, for $f, g\in \mathbf{L}\backslash
  \left\{\bm{1}_{\{1\}}\right\}$,
  $$
  (f\wustar g)(1)=
  \begin{cases}
  1, & f(1)\wedge g(1)=1, \\
  0, & f(1)\wedge g(1)<1.
  \end{cases}
  $$
\end{remark}

\begin{lemma}\label{xin-lemma}
Let $f, g\in \mathbf{L}$. Then,
\begin{itemize}
  \item[{\rm (1)}] $(f\wustar g)^{L}=(f\sqcap g)^{L}$;
  \item[{\rm (2)}] if $f(1)\wedge g(1)<1$ and $f, g\in \mathbf{L}\backslash \left\{\bm{1}_{\{1\}}\right\}$,
  $$
  (f\wustar g)^{R}(x)=
  \begin{cases}
  (f\sqcap g)^{R}(x), & x\in [0, 1), \\
  0, & x=1,
  \end{cases}
  $$
  \item[{\rm (3)}] if $f(1)\wedge g(1)=1$ and $f, g\in \mathbf{L}\backslash \left\{\bm{1}_{\{1\}}\right\}$,
  $(f\wustar g)^{R}=(f\sqcap g)^{R}$.
\end{itemize}
\end{lemma}

\begin{proof}
(1) From Definition~\ref{bingstar-operation-2}, it suffices to check that, for $f, g\in \mathbf{L}$$\backslash$$\left\{\bm{1}_{\{1\}}\right\}$
with $f(1)\wedge g(1)<1$, $(f\wustar g)^{L}(1)=(f\sqcap g)^{L}(1)$. Since $f\sqcap g\in \mathbf{L}$,
one has $(f\sqcap g)^{L}(1)=1$. From the definition of $\barwedge$, it is clear that $(f\wustar g)^{L}(1)\geq
(f\barwedge g)^{L_{\mathrm{w}}}(1)=(f\sqcap g)^{L_{\mathrm{w}}}(1)$. It follows from $f(1)\wedge g(1)<1$
that $f^{L_{\mathrm{w}}}(1)\vee g^{L_{\mathrm{w}}}(1)=1$. This, together with Lemma~\ref{L-R-Operation} and
Proposition~\ref{Pro-L}, implies that
\begin{align*}
(f\wustar g)^{L}(1)&\geq(f\sqcap g)^{L_{\mathrm{w}}}(1)=\sup_{0\leq x<1}\{(f\sqcap g)^{L}(x)\}\\
&=\sup_{0\leq x<1}\{f^{L}(x)\vee g^{L}(x)\}\geq f^{L_{\mathrm{w}}}(1)\vee g^{L_{\mathrm{w}}}(1)=1.
\end{align*}
Thus, $(f\wustar g)^{L}(1)=(f\sqcap g)^{L}(1)=1$.

(2) Fix $f, g\in \mathbf{L}\backslash\left\{\bm{1}_{\{1\}}\right\}$ with $f(1)\wedge g(1)<1$ 
and let
$\xi_1=\sup\{x\in I: f^{R}(x)=1\}$ and $\xi_2=\sup\{x\in I: g^{R}(x)=1\}$. It is clear that
$(f\wustar g)^{R}(1)=0$ since $(f\wustar g)(1)=0$. For $\hat{x}\in [0, 1)$, consider the following two cases:

Case 1. If $\xi_1\wedge \xi_2<1$, applying Lemma~\ref{sup{f^R}}, it follows that $\sup\{x\in I: (f\sqcap g)^{R}(x)=1\}=\xi_1\wedge \xi_2<1$.
This, together with Corollary~\ref{f^{R}}, implies that
\begin{align*}
(f\sqcap g)^{R}(\hat{x})&=\sup\{(f\sqcap g)(y): \hat{x}\leq y<1\}\\
&=\sup\{(f\barwedge g)(y): \hat{x}\leq y\leq 1\}\ (\text{as } (f\barwedge g)(1)=0)\\
&=(f\wustar g)^{R}(\hat{x});
\end{align*}

Case 2. If $\xi_1\wedge \xi_2=1$, i.e., $\xi_1=\xi_2=1$, then $f^{R}(\hat{x})=g^{R}(\hat{x})=1$. This, together with Lemma~\ref{L-R-Operation},
implies that $(f\sqcap g)^{R}(\hat{x})=f^{R}(\hat{x})\wedge g^{R}(\hat{x})=1$. Applying $\xi_1=\xi_2=1$ and Proposition~\ref{f^l--f^r=1}
yields that, for $z\in [0, 1)$,
$f^{R}(z)=g^{R}(z)=1$, implying that $f(z)=f^{L}(z)\wedge f^{R}(z)=f^{L}(z)$ and $g(z)=g^{L}(z)\wedge g^{R}(z)=g^{L}(z)$. Thus,
$$
(f\wustar g)(x)=
\begin{cases}
(f\sqcap g)(x)=f^{L}(x)\vee g^{L}(x), & x\in [0, 1), \\
0, & x=1.
\end{cases}
$$
Noting that $f^L$ and $g^{L}$ are increasing, by applying Proposition~\ref{Pro-L}, one has
$$
(f\wustar g)^{R}(\hat{x})=\sup\{f^{L}(y)\vee g^{L}(y): \hat{x}\leq y<1\}=f^{L_{\mathrm{w}}}(1)\vee g^{L_{\mathrm{w}}}(1)=1
=(f\sqcap g)^{R}(\hat{x}).
$$

(3) From Definition~\ref{bingstar-operation-2}, these hold trivially.
\end{proof}

\begin{proposition}\label{Closed-Lemma}
For $f, g\in \mathbf{L}$, $f\wustar g$ is normal and convex, i.e., $f\wustar g\in \mathbf{L}$.
\end{proposition}

\begin{proof}
By applying Definition~\ref{bingstar-operation-2} and Lemma~\ref{xin-lemma}, this can be verified immediately.
\end{proof}

\begin{remark}\label{R-24}
Proposition~\ref{Closed-Lemma} shows that the binary operation $\wustar$ is closed on $\mathbf{L}$, i.e.,
$\wustar(\mathbf{L}^2) \subset \mathbf{L}$.
\end{remark}

\begin{corollary}\label{f-bar-g-R}
For $f, g\in \mathbf{L}$, $R^{1}(f\wustar g)=R^{1}(f) \wedge R^{1}(g)$.
\end{corollary}

\begin{proof}
By applying Definition~\ref{bingstar-operation-2} and Lemmas~\ref{sup{f^R}} and \ref{xin-lemma}, this can be verified immediately.
\end{proof}

\begin{proposition}\label{f(x=1)}
For $f, g\in \mathbf{L}\backslash\left\{\bm{1}_{\{1\}}\right\}$, $(f \wustar g)(1)=1$ if and only if $f(1)\wedge g(1)=1$.
\end{proposition}

\begin{proof}
From Remark~\ref{Remark-bingstar}, this holds trivially.
\end{proof}

\subsection{$\wustar$ satisfies (O1)}\label{Sub-Sec-4.1}
For $f, g\in \mathbf{L}$,

A-1) if $f=\bm{1}_{\{1\}}$ or $g=\bm{1}_{\{1\}}$, then clearly $f\wustar g=g\wustar f$;

A-2) if $f\neq \bm{1}_{\{1\}}$ and $g\neq \bm{1}_{\{1\}}$, then
$$
f\wustar g=\begin{cases}
f\sqcap g, & f(1)\wedge g(1)=1,\\
f\barwedge g, & f(1)\wedge g(1)<1.
\end{cases}
$$
Meanwhile, it can be verified that
$$
g\wustar f=\begin{cases}
g\sqcap f, & g(1)\wedge f(1)=1, \\
g\barwedge f, & g(1)\wedge f(1)<1.
\end{cases}
$$
Thus, $f\wustar g=g\wustar f$.

\begin{lemma}\label{Increasing-Lemma}
Let $f, g\in \mathbf{L}$. Then, $f\wustar g \sqsubseteq f$ and $f\wustar g\sqsubseteq g$. In particular, $f\wustar g\neq \bm{1}_{\{1\}}$
if $f, g\in \mathbf{L}\backslash \left\{\bm{1}_{\{1\}}\right\}$.
\end{lemma}
\begin{proof}
Since $\wustar$ satisfies (O1), it suffices to check that $f\wustar g \sqsubseteq f$.
Consider the following three cases:

Case 1. If $f=\bm{1}_{\{1\}}$, then $f\wustar g=g\sqsubseteq \bm{1}_{\{1\}}=f$;

Case 2. If $g=\bm{1}_{\{1\}}$, then $f\wustar g=f \sqsubseteq f$;

Case 3. If $f\neq \bm{1}_{\{1\}}$ and $g\neq \bm{1}_{\{1\}}$, from (\ref{xin-operation}), Lemmas~\ref{xin-lemma},
  and $f\sqcap g\sqsubseteq f$, it follows that $(f\wustar g)^{L}=(f\sqcap g)^{L}\geq f^{L}$ and
  $(f\wustar g)^{R}\leq (f\sqcap g)^{R}\leq f^{R}$. This, together with Proposition~\ref{WW-L-function}, implies that
  $f\wustar g \sqsubseteq f$.
\end{proof}

\subsection{$\wustar$ satisfies (O2)}\label{S-4B}
For $f, g, h\in \mathbf{L}$,

B-1) if one of $f$, $g$, and $h$ is equal to $\bm{1}_{\{1\}}$, then it is easy to verify that $(f\wustar g) \wustar h
  =f\wustar (g\wustar h)$;

B-2) if none of $f$, $g$, and $h$ are equal to $\bm{1}_{\{1\}}$, from Lemmas~\ref{L-R-Operation} and \ref{xin-lemma},
it follows that
$$
((f\wustar g)\wustar h)^{L}=((f\wustar g) \sqcap h)^{L}
=(f\wustar g)^{L}\vee h^{L}= f^{L}\vee g^{L} \vee h^{L},
$$
$$
(f\wustar (g\wustar h))^{L}=(f\sqcap (g\wustar h))^{L}
=f^{L}\vee (g\wustar h)^{L}= f^{L}\vee g^{L} \vee h^{L},
$$
and, for $x\in [0, 1)$,
\begin{align*}
((f\wustar g) \wustar h)^{R}(x)&=((f\wustar g)\sqcap h)^{R}(x)=
(f\wustar g)^{R}(x) \wedge h^{R}(x)\\
&=(f\sqcap g)^{R}(x)\wedge h^{R}(x)
=f^{R}(x)\wedge g^{R}(x) \wedge h^{R}(x),
\end{align*}
\begin{align*}
(f\wustar (g\wustar h))^{R}(x)&=(f \sqcap (g\wustar h))^{R}(x)=f^{R}(x)
\wedge (g\wustar h)^{R}(x)\\
&=f^{R}(x)\wedge (g\sqcap h)^{R}(x)=
f^{R}(x)\wedge g^{R}(x) \wedge h^{R}(x).
\end{align*}
These imply that
$$
((f\wustar g)\wustar h)^{L}=(f\wustar (g\wustar h))^{L},
$$
and, for $x\in [0, 1)$,
$$
((f\wustar g) \wustar h)^{R}(x)=(f\wustar (g\wustar h))^{R}(x).
$$

To prove $(f\wustar g) \wustar h=f\wustar (g\wustar h)$, by applying Proposition~\ref{WW-L-function}--(6)
and Proposition~\ref{Closed-Lemma}, it suffices to check that $((f\wustar g) \wustar h)^{R}(1)
=(f\wustar (g\wustar h))^{R}(1)$. From Remark~\ref{Remark-bingstar}, Lemma~\ref{Increasing-Lemma}, and proposition~\ref{f(x=1)},
it follows that
\begin{equation}\label{Eq-5}
\begin{split}
&\quad((f\wustar g)\wustar h)(1)\\
&=
\begin{cases}
1, & (f\wustar g)(1)\wedge h(1)=1, \\
0, & (f\wustar g)(1)\wedge h(1)<1,
\end{cases}\\
&=
\begin{cases}
1, & f(1)\wedge g(1)\wedge h(1)=1, \\
0, & f(1)\wedge g(1)\wedge h(1)<1,
\end{cases}
\end{split}
\end{equation}
and
\begin{equation}\label{Eq-6}
\begin{split}
&\quad (f\wustar (g\wustar h))(1)\\
&=
\begin{cases}
1, & f(1)\wedge (h\wustar g)(1)=1, \\
0, & f(1)\wedge (h\wustar g)(1)<1.
\end{cases}\\
&=
\begin{cases}
1, & f(1)\wedge g(1)\wedge h(1)=1, \\
0, & f(1)\wedge g(1)\wedge h(1)<1.
\end{cases}
\end{split}
\end{equation}
Thus, $(f\wustar g) \wustar h=f\wustar (g\wustar h)$.

\subsection{$\wustar$ satisfies (O3)}

This follows directly from Cases 1 and 2 of Definition~\ref{bingstar-operation-2}.

\subsection{$\wustar$ satisfies (O4$^{'}$)}

For $f, g, h\in \mathbf{L}$, a claim is that $f\wustar (g\sqcup h)=(f\wustar g) \sqcup (f\wustar h)$.
In fact, the following are true:

D-1) If $f=\bm{1}_{\{1\}}$, then $f\wustar (g\sqcup h)=g\sqcup h=(\bm{1}_{\{1\}}\wustar g) \sqcup (\bm{1}_{\{1\}}\wustar h)
=(f\wustar g) \sqcup (f\wustar h)$.

D-2) If $g=\bm{1}_{\{1\}}$, then $f\wustar (g\sqcup h)=f\wustar \bm{1}_{\{1\}}=f$. From Lemma~\ref{Increasing-Lemma}, it follows that
$f\wustar h\sqsubseteq f$. This implies that $(f\wustar g) \sqcup (f\wustar h)=f \sqcup (f\wustar h)=f$. Thus,
$f\wustar (g\sqcup h)=(f\wustar g) \sqcup (f\wustar h)$.

D-3) If $h=\bm{1}_{\{1\}}$, since $\sqcup$ is commutative, applying D-2) yields that $f\wustar (g\sqcup h)=f\wustar (h\sqcup g)
=(f\wustar h)\sqcup (f\wustar g)=(f\wustar g) \sqcup (f \wustar h)$.

D-4) If $f\neq\bm{1}_{\{1\}}$, $g\neq \bm{1}_{\{1\}}$, and $h\neq \bm{1}_{\{1\}}$, applying Lemmas~\ref{L-R-Operation}
and \ref{xin-lemma}, it can be verified that
$$
(f\wustar (g\sqcup h))^{L}=((f\wustar g) \sqcup (f\wustar h))^{L},
$$
and, for $x\in [0, 1)$,
$$
(f\wustar (g\sqcup h))^{R}(x)=((f\wustar g) \sqcup (f\wustar h))^{R}(x).
$$
To prove that $f\wustar (g\sqcup h)=(f\wustar g) \sqcup (f\wustar h)$, applying Proposition~\ref{WW-L-function}--(6) and Remark~\ref{remark-1},
it suffices to check that $(f\wustar (g\sqcup h))(1)=((f\wustar g) \sqcup (f\wustar h))(1)$.

Applying Remarks~\ref{remark-1} and \ref{Remark-bingstar} yields that
\begin{align*}
&\quad (f\wustar (g\sqcup h))(1)\\
&=\begin{cases}
0, & f(1)\wedge (g\sqcup h)(1)<1,\\
1, & f(1)\wedge (g\sqcup h)(1)=1,
\end{cases}\\
&=\begin{cases}
0, & f(1)\wedge (g(1)\vee h(1))<1,\\
1, & f(1)\wedge (g(1)\vee h(1))<1,
\end{cases}
\end{align*}
and
\begin{align*}
&\quad ((f\wustar g) \sqcup (f\wustar h))(1)\\
&=\begin{cases}
0, & (f\wustar g)(1)\vee (f\wustar h)(1)<1,\\
1, & (f\wustar g)(1)\vee (f\wustar h)(1)=1,
\end{cases}\\
&=\begin{cases}
0, & (f(1)\wedge g(1))\vee (f(1)\wedge h(1))<1,\\
1, & (f(1)\wedge g(1))\vee (f(1)\wedge h(1))=1,
\end{cases}
\end{align*}
Therefore, $f\wustar (g\sqcup h)=(f\wustar g) \sqcup (f\wustar h)$.

\subsection{$\wustar$ satisfies (O4$^{''}$)}

For $f, g, h\in \mathbf{L}$, a claim is that $f\wustar (g\sqcap h)=(f\wustar g) \sqcap (f\wustar h)$.
In fact, the following are true:

E-1)If $f=\bm{1}_{\{1\}}$, then $f\wustar (g\sqcap h)=g\sqcap h=(\bm{1}_{\{1\}}\wustar g) \sqcap (\bm{1}_{\{1\}}\wustar h)
=(f\wustar g)\sqcap (f\wustar h)$.

E-2) If $g=\bm{1}_{\{1\}}$, then $f\wustar (g\sqcap h)=f\wustar h$. From Lemma~\ref{Increasing-Lemma}, it follows that
$f\wustar h\sqsubseteq f$. This implies that $(f\wustar g) \sqcap (f\wustar h)=f \sqcap (f\wustar h)=f\wustar h$.
Thus, $f\wustar (g\sqcap h)=(f\wustar g) \sqcap (f\wustar h)$.

E-3) If $h=\bm{1}_{\{1\}}$, since $\sqcap$ is commutative, applying E-2) yields that $f\wustar (g\sqcap h)=f\wustar (h\sqcap g)
=(f\wustar h)\sqcap (f\wustar g)=(f\wustar g) \sqcap (f \wustar h)$.

E-4) If $f\neq\bm{1}_{\{1\}}$, $g\neq \bm{1}_{\{1\}}$, and $h\neq \bm{1}_{\{1\}}$, applying Lemmas~\ref{L-R-Operation}
and \ref{xin-lemma}, it can be verified that
$$
(f\wustar (g\sqcap h))^{L}=((f\wustar g) \sqcap (f\wustar h))^{L},
$$
and, for $x\in [0, 1)$,
$$
(f\wustar (g\sqcap h))^{R}(x)=((f\wustar g) \sqcap (f\wustar h))^{R}(x).
$$
To prove that $f\wustar (g\sqcap h)=(f\wustar g) \sqcap (f\wustar h)$, applying Proposition~\ref{WW-L-function}--(6), it suffices to
check that $(f\wustar (g\sqcap h))(1)=((f\wustar g) \sqcap (f\wustar h))(1)$.

Applying Remarks~\ref{remark-1} and \ref{Remark-bingstar} yields that
\begin{align*}
&\quad (f\wustar (g\sqcap h))(1)\\
&=\begin{cases}
0, & f(1)\wedge (g\sqcap h)(1)<1,\\
1, & f(1)\wedge (g\sqcap h)(1)=1,
\end{cases}\\
&=\begin{cases}
0, & f(1)\wedge g(1)\wedge h(1)<1,\\
1, & f(1)\wedge g(1)\wedge h(1)<1,
\end{cases}
\end{align*}
and
\begin{align*}
&\quad ((f\wustar g) \sqcap (f\wustar h))(1)\\
&=\begin{cases}
0, & (f\wustar g)(1)\wedge (f\wustar h)(1)<1,\\
1, & (f\wustar g)(1)\wedge (f\wustar h)(1)=1,
\end{cases}\\
&=\begin{cases}
0, & (f(1)\wedge g(1))\wedge (f(1)\wedge h(1))<1,\\
1, & (f(1)\wedge g(1))\wedge (f(1)\wedge h(1))=1,
\end{cases}
\end{align*}
Thus,
$f\wustar (g\sqcap h)=(f\wustar g) \sqcap (f\wustar h)$.

\subsection{$\wustar$ satisfies (O5)}

For $0\leq a\leq b\leq 1$,

F-1) if $a=1$, then $\bm{1}_{[0, 1]}\wustar \bm{1}_{[a, b]}=\bm{1}_{[0, 1]}\wustar \bm{1}_{\{1\}}=\bm{1}_{[0, 1]}$;

F-2) if $a<1$, then consider the following two cases:
\begin{enumerate}[(i)]
  \item $b<1$, then $\bm{1}_{[0, 1]}(1)\wedge \bm{1}_{[a, b]}(1)=0<1$. This, together with Definition~\ref{bingstar-operation-2},
  implies that
  $$
  \bm{1}_{[0, 1]}\wustar \bm{1}_{[a, b]}=\bm{1}_{[0, 1]}\barwedge \bm{1}_{[a, b]}.
  $$
  Since $(\bm{1}_{[0, 1]}\sqcap \bm{1}_{[a, b]})(1)=0$, one has
  $$
  \bm{1}_{[0, 1]}\wustar \bm{1}_{[a, b]}=\bm{1}_{[0, 1]}\sqcap \bm{1}_{[a, b]}=\bm{1}_{[0, b]};
  $$
  \item $b=1$, then $\bm{1}_{[0, 1]}(1)\wedge \bm{1}_{[a, b]}(1)=1$. This, together with Definition~\ref{bingstar-operation-2},
  implies that
  $$
\bm{1}_{[0, 1]}\wustar \bm{1}_{[a, b]}=\bm{1}_{[0, 1]}\sqcap \bm{1}_{[a, b]}=
\bm{1}_{[0, b]}.
$$
\end{enumerate}

\subsection{$\wustar$ satisfies (O6)}
For $x_1, x_2\in I$, consider the following two cases:

Case 1. If $x_1=1$ or $x_2=1$, it is clear that $\bm{1}_{\{x_1\}}\wustar \bm{1}_{\{x_2\}}=\bm{1}_{\{x_2\}}\wustar \bm{1}_{\{x_1\}}\in \mathbf{J}$.

Case 2. If $x_1\neq 1$ and $x_2\neq 1$, from Definition~\ref{bingstar-operation-2}, it can be verified that
$\bm{1}_{\{x_1\}}\wustar \bm{1}_{\{x_2\}}=\bm{1}_{\{x_2\}}\wustar \bm{1}_{\{x_1\}}=\bm{1}_{\{x_1\}}\sqcap \bm{1}_{\{x_2\}}=\bm{1}_{\{x_1\wedge x_2\}}\in \mathbf{J}$.

\subsection{$\wustar$ satisfies (O7)}\label{Sub-Sec-4.7}
For $[a_1, b_1], [a_2, b_2]\subset I$ with $[a_1, b_1]\neq \{1\}$ and $[a_2, b_2]\neq \{1\}$,
from Definition~\ref{bingstar-operation-2}, it follows that
\begin{align*}
&\quad \bm{1}_{[a_1, b_1]}\wustar \bm{1}_{[a_2, b_2]}\\
&=
\begin{cases}
\bm{1}_{[a_1, b_1]}\sqcap \bm{1}_{[a_2, b_2]}, & b_1=1 \text{ and } b_2=1, \\
\bm{1}_{[a_1, b_1]}\barwedge \bm{1}_{[a_2, b_2]}, & b_1<1 \text{ or } b_2<1,
\end{cases}\\
&=
\begin{cases}
\bm{1}_{[a_1\wedge a_2, b_1\wedge b_2]}, & b_1=1 \text{ and } b_2=1, \\
\bm{1}_{[a_1\wedge a_2, b_1\wedge b_2]}, & b_1<1 \text{ or } b_2<1,
\end{cases}\\
&=\bm{1}_{[a_1\wedge a_2, b_1\wedge b_2]}\in \mathbf{K}.
\end{align*}
This, together with the commutativity of $\wustar$, implies that
$$
\bm{1}_{[a_1, b_1]}\wustar \bm{1}_{[a_2, b_2]}=\bm{1}_{[a_2, b_2]}\wustar \bm{1}_{[a_1, b_1]}=
\bm{1}_{[a_1\wedge a_2, b_1\wedge b_2]}\in \mathbf{K}.
$$
Clearly, $\bm{1}_{[a_1, b_1]}\wustar \bm{1}_{[a_2, b_2]}=\bm{1}_{[a_2, b_2]}\wustar \bm{1}_{[a_1, b_1]}\in \mathbf{K}$
when $[a_1, b_1]=\{1\}$ or $[a_2, b_2]=\{1\}$.

\medskip

Combining \ref{Sub-Sec-4.1}--\ref{Sub-Sec-4.7} together immediately yields the following result.

\medskip

\begin{theorem}\label{t{lor}-norm-theorem}
The binary operation $\wustar$ is a $t_{lor}$-norm on $\mathbf{L}$. In particular,
$\wustar$ is a $t_{r}$-norm on $\mathbf{L}$.
\end{theorem}

\section{$\wustar$ cannot be obtained by $\curlywedge$}\label{S-V}

This section proves that the $t_{lor}$-norm $\wustar$ constructed in Section~\ref{S-iV}
cannot be obtained by operations $\curlywedge$. This shows that the $t_{lor}$-norm $\wustar$
is not the convolution of each $t$-norm on $I$, answering negatively Question~\ref{Q-2}.

The following theorem provides a sufficient condition ensuring that $\ast$ is a $t$-norm on $I$.

\begin{theorem}{\rm\cite[Theorem~12]{WC-T-Norm}}\label{WC-T-Norm}
Let $\ast$ be a binary operation on $I$ and $\vartriangle$ be a $t$-norm on $I$.
If the binary operation $\curlywedge$ is a $t_r$-norm on $\mathbf{L}$, then
$\vartriangle$ is a continuous $t$-norm and $\ast$ is a $t$-norm.
\end{theorem}

\begin{proposition}\label{1-Lemma}
Let $\ast$ be a $t$-norm on $I$. Then, $x\ast y=1$ if and only if $x=y=1$.
\end{proposition}

\begin{theorem}\label{Thm-30}
For any binary operation $\ast$ on $I$ and any $t$-norm $\vartriangle$ on $I$,
there exist $f, g\in \mathbf{L}$ such that $f\wustar g \neq f\curlywedge g$. In particular,
$\wustar$ is not the convolution of each $t$-norm on $I$.
\end{theorem}

\begin{proof}
Suppose, on the contrary, that there exist a binary operation $\ast$ on $I$ and a $t$-norm $\vartriangle$ on $I$
such that, for $f, g\in \mathbf{L}$, one has $f\wustar g=f\curlywedge g$. Applying Theorem~\ref{WC-T-Norm},
yields that $\ast$ is a $t$-norm on $I$.
Choose $f=\bm{1}_{[0, 1]}$ and
$$
g(x)=
\begin{cases}
2x, & x\in [0, 0.5], \\
-x+1.5, & x\in (0.5, 1].
\end{cases}
$$
Clearly, $f, g\in \mathbf{L}$. From Definition~\ref{bingstar-operation-2}, it is easy to see that
$(f\wustar g)(1)=0$, since $f(1)\wedge g(1)<1$. This, together Theorem~\ref{t{lor}-norm-theorem}
and Proposition~\ref{1-Lemma}, implies that
$$
(f\curlywedge g)(1)=f(1)\ast g(1)=1\ast 0.5=0.5\neq (f\wustar g)(1),
$$
which contradicts with $f\wustar g=f\curlywedge g$.
\end{proof}

\begin{remark}
Combining Theorems \ref{t{lor}-norm-theorem} and \ref{Thm-30} negatively answers Question~\ref{Q-2}.
\end{remark}

\section{Conclusion}
Continuing our study in \cite{WC-TFS-2,WC-T-Norm}, this paper constructs two binary operations $\circledast$ and $\wustar$
on $I^{[2]}$ and $\mathbf{L}$, respectively (see Definitions \ref{xinxing-operation} and \ref{bingstar-operation-2}), and proves that
\begin{enumerate}[(i)]
  \item the binary operation $\circledast$ satisfies conditions (4) and (5) in Definition~\ref{Def-2}, but
  does not satisfy condition (5$^{\prime}$);
  \item the binary operation $\wustar$ is a $t_{lor}$-norm on $\mathbf{L}$, but not the convolution of any $t$-norms on $I$.
\end{enumerate}
These two results negatively answer Questions \ref{Q-1} and \ref{Q-2}
originally posed by Walker and Walker in \cite{WW2006}.


\section*{References}
\begin{thebibliography}{10}
\bibitem{GWW1996}
M.~Gehrke, C.~L. Walker, and E.~A. Walker.
\newblock Some comments on interval valued fuzzy sets.
\newblock {\em Internat. J. Intell. Systems}, 11:751--759, 1996.

\bibitem{HWW2007}
J.~Harding, C.~L. Walker, and E.~A. Walker.
\newblock On complete sublattices of the algebra of truth values of type-2
  fuzzy sets.
\newblock In {\em 2007 IEEE Int. Fuzzy Syst. Conf.}, pages 1--5. IEEE, 2007.

\bibitem{HWW2008}
J.~Harding, C.~L. Walker, and E.~A. Walker.
\newblock Lattices of convex normal functions.
\newblock {\em Fuzzy Sets Syst.}, 159:1061--1071, 2008.

\bibitem{HWW2010}
J.~Harding, C.~L. Walker, and E.~A. Walker.
\newblock Convex normal functions revisited.
\newblock {\em Fuzzy Sets Syst.}, 161:1343--1349, 2010.

\bibitem{HWW2016}
J.~Harding, C.~L. Walker, and E.~A. Walker.
\newblock {\em The {T}ruth {V}alue {A}lgebra of {T}ype-2 {F}uzzy {S}ets:
  {O}rder {C}onvolutions of {F}unctions on the {U}nit {I}nterval}.
\newblock CRC Press, 2016.

\bibitem{HCT2015}
P.~Hern\'{a}ndez, S.~Cubillo, and C.~Torres-Blanc.
\newblock On $t$-norms for type-2 fuzzy sets.
\newblock {\em IEEE Trans. Fuzzy Syst.}, 23:1155--1163, 2015.

\bibitem{KMP2000}
E.~P. Klement, R.~Mesiar, and E.~Pap.
\newblock {\em Triangular {N}orms}, volume~8 of {\em Trends in Logic}.
\newblock Springer Netherlands, 2000.

\bibitem{M2007}
J.~M. Mendel.
\newblock Advances in type-2 fuzzy sets and systems.
\newblock {\em Inf. Sci.}, 177(1):84--110, 2007.

\bibitem{M2001}
J.~M. Mendel.
\newblock Uncertain {R}ule-{B}ased {F}uzzy {S}ystems.
\newblock In {\em Introduction and {N}ew {D}irections}. Springer; 2nd ed.,
  2017.

\bibitem{MJ2002}
J.~M. Mendel and R.~I.~B. John.
\newblock Type-2 fuzzy sets made simple.
\newblock {\em IEEE Trans. Fuzzy Syst.}, 10(2):117--127, 2002.

\bibitem{Me1942}
K.~Menger.
\newblock Statistical metrics.
\newblock {\em Proc. Nat. Acad. Sci. USA}, 37:535--537, 1942.

\bibitem{MT1976}
M.~Mizumoto and K.~Tanaka.
\newblock Some properties of fuzzy sets of type-2.
\newblock {\em Inf. Control}, 31:312--340, 1976.

\bibitem{SS1961}
B.~Schweizer and A.~Sklar.
\newblock Associative functions and statistical triangle inequalities.
\newblock {\em Publ. Math.}, 8:169--186, 1961.

\bibitem{WW2005}
C.~L. Walker and E.~A. Walker.
\newblock The algebra of fuzzy truth values.
\newblock {\em Fuzzy Sets Syst.}, 149:309--347, 2005.

\bibitem{WW2006}
C.~L. Walker and E.~A. Walker.
\newblock T-norms for type-2 fuzzy sets.
\newblock In {\em 2006 IEEE Int. Conf. Fuzzy Syst.}, pages 1235--1239. IEEE,
  2006.

\bibitem{WC-TFS-2}
X.~Wu and G.~Chen.
\newblock On the existence of $t_r$-norm and $t_r$-conorm not in convolution
  form.
\newblock {\em submitted to IEEE Trans. Fuzzy Syst.}
\newblock \url{https://arxiv.org/abs/1908.10532}.

\bibitem{WC-T-Norm}
X.~Wu and G.~Chen.
\newblock Revisiting $t$-norms for type-2 fuzzy sets.
\newblock \url{https://arxiv.org/abs/2003.11953}.

\bibitem{WC-TFS}
X.~Wu and G.~Chen.
\newblock Answering an open problem on $t$-norms for type-2 fuzzy sets.
\newblock {\em Inf. Sci.}, 522:124--133, 2020.

\bibitem{Z1965}
L.~A. Zadeh.
\newblock Fuzzy sets.
\newblock {\em Inf. Control}, 20:301--312, 1965.

\bibitem{Z1975}
L.~A. Zadeh.
\newblock The concept of a linguistic variable and its application to
  approximate reasoning-{II}.
\newblock {\em Inf. Sci.}, 8(4):301--357, 1975.

\end{thebibliography}
\end{document}